\documentclass[a4paper, 12pt, reqno]{amsart}

\usepackage[english]{babel}

\usepackage{amsfonts, amsmath, amssymb, amsthm}
\usepackage[shortlabels]{enumitem}
\setlist[enumerate]{label=\normalfont{(\arabic*)}}
\usepackage[colorlinks=true, linkcolor=blue, citecolor=blue]{hyperref}
\usepackage{stmaryrd}
\usepackage{tikz-cd}

\usepackage{scalerel}
\usepackage{xifthen}

\advance\hoffset-20mm
\advance\textwidth40mm

\newcommand{\bb}{\mathbb}
\newcommand{\cal}{\mathcal}

\newcommand{\Ga}{\bb{G}_\mathrm{a}}
\newcommand{\into}{\hookrightarrow}
\newcommand{\reg}{\mathrm{reg}}

\DeclareMathOperator{\AffCone}{AffCone}
\DeclareMathOperator{\Aut}{Aut}

\DeclareMathOperator{\grOp}{Gr}
\DeclareMathOperator{\Ker}{Ker}

\DeclareMathOperator{\Pic}{Pic}

\DeclareMathOperator{\SAut}{SAut}
\DeclareMathOperator{\SOop}{SO}
\DeclareMathOperator{\Spec}{Spec}
\DeclareMathOperator{\Supp}{Supp}

\newcommand{\gen}[1]{\left\langle #1 \right\rangle}
\newcommand{\Gr}[2]{\grOp \left( {#1}, {#2} \right)}
\newcommand{\insertText}[2][]{%
    \ifthenelse{\isempty{#1}}{%
        \quad \text{#2} \quad%
    }{%
        #1 \text{#2} #1%
    }%
}

\newcommand{\restrict}[2]{
    \left.
    \kern-\nulldelimiterspace{#1}
    \right|_{#2}
}
\newcommand{\SO}[1]{\SOop_{#1}(\mathbb{K})}
\newcommand{\wh}[1]{
    \hstretch{2}{\hat{\hstretch{.5}{#1}}}
}

\newtheorem{counter}{}[section]

\theoremstyle{definition}
\newtheorem{defi}[counter]{Definition}

\theoremstyle{plain}
\newtheorem{coro}[counter]{Corollary}
\newtheorem{lemm}[counter]{Lemma}
\newtheorem{prop}[counter]{Proposition}
\newtheorem{theo}[counter]{Theorem}

\theoremstyle{remark}
\newtheorem{rema}[counter]{Remark}

\begin{document}


\date{}
\title[Flexibility of affine cones]{Flexibility of affine cones over a smooth complete intersection of two quadrics}

\author[Kirill Shakhmatov]{Kirill Shakhmatov}
\address{HSE University, Faculty of Computer Science, Pokrovsky Boulevard 11, Moscow, 109028 Russia}
\email{kshahmatov@hse.ru}

\author[Hoang Le Truong]{Hoang Le Truong}
\address{Institute for Artificial Intelligence, University of Engineering and Technology, Vietnam National University, Hanoi, Viet Nam}
\email{hltruong@vnu.edu.vn\\truonghoangle@gmail.com}

\thanks{The article was prepared within the framework of the project ``International Academic Cooperation'' HSE~University}



\begin{abstract}
We prove flexibility of two families of affine varieties: the complement in $\bb{P}^n$ of a projective quadric of rank at least three and affine cones over a smooth complete intersection of two quadrics in $\bb{P}^{n + 2}$, $n \ge 3$.
\end{abstract}

\maketitle


\section{Introduction}


The geometry of automorphism groups has long offered a revealing lens through which to view the structure of algebraic varieties. In the affine setting, varieties that admit an abundance of unipotent transformations exhibit particularly rich behaviour. Among them, the \emph{flexible varieties} form a distinguished class: they are those for which every tangent direction at a smooth point arises from an additive group action. Such varieties can be regarded as algebraically homogeneous in the strongest possible sense.

Throughout we work over an algebraically closed field \( \bb K \) of characteristic zero. A \emph{variety} means an integral separated scheme of finite type over \( \bb K \). We write \( \Ga = (\bb K, +) \) for the one-dimensional additive algebraic group. For a variety \( X \), a point \( x \in X \) is said to be \emph{flexible} if the tangent space \( T_xX \) is spanned by tangent vectors to orbits of effective \( \Ga \)-actions on \( X \); the variety \( X \) is \emph{flexible} if every smooth point is flexible. Denote by \( \operatorname{SAut}(X) \subseteq \operatorname{Aut}(X) \) the subgroup generated by all one-parameter unipotent subgroups of automorphisms of \( X \). A fundamental theorem of Arzhantsev, Flenner, Kaliman, Kutzschebauch, and Zaidenberg~\cite{AFKKZ13} shows that flexibility is equivalent to a strong homogeneity property of this group:  

\begin{theo}[{\cite[Theorem~0.1]{AFKKZ13}}] \label{FIT.th}
Let \( X \) be an affine variety. Then the following are equivalent:  
\begin{enumerate}
  \item \( X \) is flexible;  
  \item \( \SAut(X) \) acts transitively on the smooth locus \( X^\reg \);  
  \item if \( \dim X \ge 2 \), \( \SAut(X) \) acts infinitely transitively on \( X^\reg \).
\end{enumerate}
\end{theo}

This theorem---widely regarded as a cornerstone in the development of the theory---places the notion of flexibility at the intersection of algebraic transformation groups and birational geometry. It identifies a distinguished class of affine varieties whose automorphism groups behave much like those of homogeneous spaces.

The idea of flexibility first appeared in the study of affine toric varieties, where the presence of additive group actions can be described combinatorially in terms of the defining fan \cite{AZK12}. Over time, the concept proved far more general and has become deeply intertwined with the geometry of \emph{Fano varieties} and their \emph{affine cones}. Given a projective variety \( X \) and a very ample divisor \( H \) on \( X \), one can form the affine cone $\AffCone_H(X)$ over the embedding $X \into \bb{P}^N$ defined by the complete linear system $|H|$. A natural question, raised a little over a decade ago, asks for which projective \emph{rational} varieties \( X \) this affine cone is flexible. This problem has gradually shaped the modern theory of flexible varieties, linking unipotent group actions with the geometry of Fano varieties.

The first decisive progress was made by Perepechko~\cite{Pe13}, who classified flexible affine cones over del Pezzo surfaces: they occur precisely for degrees $\ge 4$, while cones over surfaces of smaller degree fail to be flexible. This revealed, already in dimension two, a sharp boundary between flexible and rigid behaviour. In higher dimensions, Prokhorov and Zaidenberg~\cite{PZ23} established flexibility for the affine cones over Fano--Mukai fourfolds of genus~10, and later Hang and the second author~\cite{HT24} extended the result to fourfolds of genus~7 (cf. \cite{HT22, HHT22}). The work of Cheltsov, Park, Prokhorov, and Zaidenberg~\cite{CPPZ21, KPZ11, KPZ13} developed a geometric framework based on \emph{cylinders}---open subsets isomorphic to products with an affine space---and showed that a transversal covering by such cylinders often guarantees flexibility. Together, these results connected flexibility with core geometric structures of Fano varieties, including families of lines, birational projections, and degenerations of hyperplane sections.

Within this evolution, a natural and historically significant next step concerns the class of \emph{smooth complete intersections of two quadrics}. These varieties occupy a central place in algebraic geometry, appearing in the early classification of Fano manifolds and re-emerging in modern studies of birational rigidity, derived categories, and moduli (\cite{Re72, HT21, HT25}). For \( n \ge 3 \), such a variety  
\begin{equation*}
X = Q_1 \cap Q_2 \subseteq \bb{P}^{n + 2}
\end{equation*}
is a Fano variety of Picard number~1 and index~\( n-1 \); in the case \( n = 2 \), one recovers the del Pezzo surface of degree~4. Despite their classical nature, the behaviour of their affine cones under unipotent automorphisms remained open. The present paper closes this gap by proving that the affine cones over all such varieties are flexible, thereby extending the known frontier of flexibility to one of the most fundamental families in Fano geometry.

\begin{theo}[Main Theorem]
Let \( X \subseteq \bb{P}^{n + 2} \) be a smooth complete intersection of two quadrics with \( n \ge 3 \). Then the affine cone over \( X \) with respect to any very ample divisor is flexible.
\end{theo}

The proof is geometric in spirit. We begin by establishing the flexibility of the complement in \( \bb{P}^n \) of a quadric of rank at least three, achieved by constructing an explicit transversal covering by principal cylinders. We then employ a birational model due to Kishimoto, Prokhorov, and Zaidenberg that relates open subsets of \( X \) to such complements, thereby transferring flexibility to the affine cone. In this way, flexibility of the affine cone over the intersection of two quadrics emerges as a direct consequence of the geometry of its ambient projective space.

The paper is organized as follows. Section 2 reviews the correspondence between $\Ga$ actions and cylindrical open subsets in affine varieties. Section 3 establishes a general criterion for flexibility in terms of transversal coverings by cylinders. Section 4 is devoted to proving the flexibility of affine cones over quadrics of rank at least three. Section 5 then shows that the complement of a quadric of rank $\geq 3$ in projective space admits a transversal cylinder covering and is therefore flexible. Finally, Section 6 applies these results to smooth complete intersections of two quadrics, completing the proof of the main theorem.

\medskip

Throughout this paper we will use the following notations and conventions.
\begin{itemize}
\item
Let $\alpha \colon G \times X \to X$ be an action of an algebraic group $G$ on a variety $X$. Given an invariant subvariety $Y \subseteq X$, we denote by $\restrict{\alpha}{Y}$ the induced action $G \times Y \to Y$. Given an automorphism $\phi \in \Aut(X)$, by the action on $X$ conjugated to $\alpha$ by $\phi$ we mean the action $\beta \colon G \times X \to X$ defined by the formula
$$
\beta(g, x) =
\phi \Big( \alpha \big( g, \phi^{-1}(x) \big) \Big).
$$


\item
Let $X$ be a variety. We will use the notations
$$
\bb{D}_X(h) = \{ x \in X \mid h(x) \ne 0 \} \subseteq X
\insertText{and}
\bb{V}_X(h) = X \setminus \bb{D}_X(h)
$$
in two cases: either $h \in \bb{K}(X)$ is a rational function, or $X \subseteq \bb{P}^n$ is a projective variety and $h \in H^0 \big( \bb{P}^n, \cal{O}_{\bb{P}^n}(m) \big)$, $m \ge 0$ is a homogeneous polynomial.


\item
Given a sequence $(x_1, \dots, x_n)$, we denote
$$
(x_1, \dots, \hat{x}_i, \dots, x_n) =
(x_1, \dots, x_{i - 1}, x_{i + 1}, \dots, x_n).
$$


\item
Given a projective variety $X \subseteq \bb{P}^n$ and a point $x \in X$, we denote by $\bb{T}_x X$ the projective tangent space of $X$ in $x$.


\item
Given a subset $W \subseteq \bb{P}^n$ in a projective space, we denote by $\gen{W}_{\bb{P}^n}$ the minimal linear subspace in $\bb{P}^n$ containing $W$.


\item
Given two non-negative integers $k \le n$, we denote by $\Gr{k}{n}$ the Grassmannian of $k$-dimensional subspaces in an $n$-dimensional $\bb{K}$-vector space.
\end{itemize}

\medskip

The authors are grateful to Alexander Perepechko, Yuri Prokhorov, and Nikita Virin for useful discussion and references.


\section{Cylinders and actions of the additive group}
\label{CAAG.se}


In this section we discuss cylinders in affine varieties and their connection with $\Ga$-actions.

\begin{defi}
Let $X$ be a variety and $k$ be a positive integer. By a \emph{$k$-cylinder in $X$} we mean an open subset $U$ together with an isomorphism $U \cong \bb{A}^k \times Z$ for some affine variety $Z$. By an \emph{$\bb{A}^k$-fiber} of a $k$-cylinder $U$ we mean a fiber of the projection $U \to Z$ to the second factor.
\end{defi}

\begin{defi}
Let $\alpha \colon \Ga^k \times X \to X$ be a $\Ga^k$-action on a variety $X$. We say that $\alpha$ is \emph{cylindrical} if there is an isomorphism $X \cong \bb{A}^k \times Z$ for some affine variety $Z$ such that $\alpha$ acts by translations along the first factor.
\end{defi}

Note that if $\alpha$ is a cylindrical $\Ga^k$-action on a variety $X \cong \bb{A}^k \times Z$, then $X$ is affine and every $\bb{A}^k$-fiber of $X$ is an orbit of $\alpha$.

Let $X$ be an affine variety. Recall the connection between principal $1$-cylinders in $X$ and effective $\Ga$-actions on $X$ from \cite[Proposition~3.1.5]{KPZ11}. If $U \subseteq X$ is a principal $1$-cylinder, then there exists an effective $\Ga$-action $\alpha$ on $X$ such that every point in $X \setminus U$ is fixed by $\alpha$ and every $\bb{A}^1$-fiber of $U$ is an orbit of $\alpha$. Conversely, if $\alpha$ is an effective $\Ga$-action on $X$, then there exists an $\alpha$-invariant $1$-cylinder $U \subseteq X$ such that $\restrict{\alpha}{U}$ is cylindrical.

\begin{lemm} \label{CGA.le}
Let $X$ be an affine variety and $\alpha$ be an effective $\Ga$-action on $X$. Assume that there is an isomorphism $X \cong \bb{A}^1 \times Z$ such that every $\bb{A}^1$-fiber is an orbit of $\alpha$. Then $\alpha$ is cylindrical.
\end{lemm}

\begin{proof}
Let $\partial$ be the locally nilpotent derivation on $\bb{K}[X] \cong \bb{K}[Z][t]$ corresponding to $\alpha$. Then, since $\alpha$ acts only along $\bb{A}^1$-fibers, we have $\bb{K}[Z] \subseteq \Ker \partial$. Since $\partial$ is locally nilpotent, it is easy to see that $\partial(t) = h$ for some $h \in \bb{K}[Z]$. Therefore,
$$
\alpha \big(s, (t, z) \big) = (t + s h, z)
\insertText{for all}
(t, z) \in \bb{A}^1 \times Z.
$$
Since $\alpha$ has no fixed points and every point of $\bb{V}_X(h)$ is fixed by $\alpha$, we conclude that $h$ is an invertible function on $X$. Consider an automorphism
$$
\phi \colon \bb{A}^1 \times Z \to \bb{A}^1 \times Z, \quad
\phi(t, z) = \left( \frac{t}{h}, z \right).
$$
Then the action $\beta$ on $X$ conjugated to $\alpha$ by $\phi$ is defined by
$$
\beta \big( s, (t, z) \big) =
\phi \Big( \alpha \big( s, (h t, z) \big) \Big) =
\phi(h t + s h, z) = (t + s, z),
$$
so that $\beta$ acts by translations along the first factor.
\end{proof}

It follows that if $U \subseteq X$ is a principal $1$-cylinder, then there exists an effective $\Ga$-action $\alpha$ on $X$ such that every point in $X \setminus U$ is fixed by $\alpha$ and $\restrict{\alpha}{U}$ is cylindrical.

\begin{defi}
Let $X$ be a variety and $U \cong \bb{A}^k \times Z$ be a $k$-cylinder in $X$. By a \emph{coordinate decomposition} of $U$ we mean a choice of coordinates $x_1, \dots, x_k$ in $\bb{A}^k$ and a collection $U_1, \dots, U_k$ of $1$-cylinders in $X$ such that for each $1 \le i \le k$ we have $U_i = U$ and $U_i \cong \bb{A}^1 \times Z_i$, where
$$
\bb{A}^1 = \Spec \bb{K}[x_i]
\insertText{and}
Z_i = \Spec \bb{K}[Z][x_1, \dots, \hat{x}_i, \dots, x_k].
$$
\end{defi}

Let $X$ be an affine variety and $U \subseteq X$ be a $k$-cylinder. Each $1$-cylinder in a coordinate decomposition of $U$ defines a $\Ga$-action on $X$ and one can easily see that these $k$ actions define an effective $\Ga^k$-action $\alpha$ on $X$ such that every point in $X \setminus U$ is fixed by $\alpha$ and $\restrict{\alpha}{U}$ is cylindrical.

\begin{defi}
Let $X$ be an affine variety. Given a principal $k$-cylinder $U \subseteq X$, we say that an effective $\Ga^k$-action $\alpha$ on $X$ is \emph{induced} by $U$ if every point in $X \setminus U$ is fixed by $\alpha$ and $\restrict{\alpha}{U}$ is cylindrical.
\end{defi}

Summing up, we see that if $U \subseteq X$ is a principal $k$-cylinder in an affine variety $X$, then there exists a $\Ga^k$-action on $X$ induced by $U$.

\begin{defi}
Let $X$ be a variety.
\begin{enumerate}
\item
Let $U \cong \bb{A}^k \times Z$ be a $k$-cylinder in $X$. Denote by $\pi \colon U \to Z$ the natural projection. A subset $S \subseteq X$ is called \emph{invariant} with respect to $U$ if
$$
S \cap U = \pi^{-1}(\pi(S \cap U)).
$$


\item
Let $\cal{U} = \{U_i\}_{i \in I}$ be a collection of cylinders in $X$. A subset $S \subseteq X$ is called \emph{invariant} with respect to $\cal{U}$ if $S$ is invariant with respect to each $U_i$, $i \in I$.
\end{enumerate}
\end{defi}

In other words, a subset $S \subseteq X$ is invariant with respect to a cylinder $U \subseteq X$, if whenever $S$ contains a point $x \in U$, it also contains the whole fiber of $U$ which contains $x$.

Let $X$ be an affine variety and $\cal{U} = \{U_i\}_{i \in I}$ be a collection of cylinders in $X$. For each $i \in I$ let $\alpha_i$ be an action on $X$ induced by $U_i$. One can easily see that a subset $S \subseteq X$ is invariant with respect to $\cal{U}$ if and only if $S$ is $\alpha_i$-invariant for all $i \in I$.


\section{Cylinders and flexibility}
\label{CF.se}


In this section we establish a criterion of flexibility of an affine variety in terms of cylinders. We also recall a sufficient condition of flexibility of an affine cone over a projective variety in terms of cylinders.

\begin{defi}
Let $X$ be a variety and $\cal{U} = \{U_i\}_{i \in I}$ be a collection of cylinders in $X$. The collection $\cal{U}$ is called \emph{transversal} if no proper subset of $\bigcup_{i \in I} U_i$ is invariant with respect to $\cal{U}$.
\end{defi}

The following lemma provides simple technical observations on transversal cylinder coverings.

\begin{lemm} \label{LTC.le}
Let $X$ be a variety.
\begin{enumerate}
\item
Let $\cal{V} = \{V_i\}_{i \in I}$ be a covering of $X$ by some open subsets. Assume that for each $i \in I$ there exists a transversal covering $\cal{U}_i = \{U_{i, j}\}_{j \in J_i}$ of $V_i$ by cylinders. Then
$$
\cal{U} = \bigcup_{i \in I} \cal{U}_i
$$
is a transversal covering of $X$.


\item
Let $\cal{U} = \{U_i\}_{i \in I}$ be a covering of $X$ by cylinders. For each $i \in I$ let
$$
\cal{V}_i = \{U_{i, 1}, \dots, U_{i, k_i}\}
$$
be a coordinate decomposition of the cylinder $U_i$. Denote $\cal{V} = \bigcup_{i \in I} \cal{V}_i$. Then $\cal{U}$ is transversal if and only if $\cal{V}$ is transversal.
\end{enumerate}
\end{lemm}

\begin{proof}
\leavevmode
\begin{enumerate}
\item
Let $S \subseteq X$ be a nonempty subset, which is invariant with respect to $\cal{U}$. Since $\cal{V}$ is a covering of $X$, there exists $i \in I$ such that the intersection $S \cap V_i$ is nonempty. One can easily see that this intersection is invariant with respect to $\cal{U}_i$, hence $S \cap V_i = V_i \subseteq S$. Since $X$ is irreducible, we conclude that $S$ intersects $V_i$ for all $i \in I$, so, by the same argument, we have $V_i \subseteq S$ for all $i \in I$. Therefore, $S = X$.


\item
The assertion follows from the fact that for each $i \in I$ a subset $S \subseteq X$ is invariant with respect to $U_i$ if and only if $S$ is invariant with respect to $\cal{V}_i$.
\end{enumerate}
\end{proof}

The following two results relate flexibility of affine varieties to existence of a transversal covering by cylinders.

\begin{lemm} \label{TCIF.le}
Let $X$ be an affine variety. Assume that there exists a transversal covering $\cal{U} = \{U_i\}_{i \in I}$ of $X$ by principal cylinders. Then $X$ is smooth and flexible.
\end{lemm}

\begin{proof}
For each $i \in I$ choose an action $\alpha_i$ on $X$ induced by $U_i$ and denote by $H_i$ the corresponding subgroup of $\Aut(X)$. Let $G$ be the subgroup of $\Aut(X)$ generated by all $H_i$. Then $G$ acts on $X$ transitively, hence $X$ is smooth. Moreover, $G \subseteq \SAut(X)$, so $X$ is flexible by Theorem~\ref{FIT.th}.
\end{proof}

\begin{prop} \label{CF.pr}
Let $X$ be an affine variety of dimension at least one. Then $X$ is flexible if and only if there exists a transversal covering of $X^\reg$ by cylinders, which are principal in $X$.
\end{prop}

\begin{proof}
The case of $\dim X = 1$ is trivial since the only affine curve which admits an effective $\Ga$-action is $\bb{A}^1$. Therefore, we may assume that $\dim X \ge 2$.

Assume that $X$ is flexible. By Theorem~\ref{FIT.th} the group $\SAut(X)$ acts transitively on $X^\reg$, hence there exists an effective $\Ga$-action on $X$. It follows that there exists a principal \linebreak
$1$-cylinder $U = \bb{D}_X(h)$ in $X$. Without loss of generality we may assume that $U \subseteq X^\reg$, see \cite[Section~1.1]{KPZ13}. Consider a collection
$$
\cal{U} = \{g U\}_{g \in \SAut(X)}
$$
of principal $1$-cylinders in $X$. Clearly, $\cal{U}$ covers $X^\reg$. Consider a nonempty subset $S \subseteq X$ which is invariant with respect to $\cal{U}$. There exists $U' \in \cal{U}$ such that $S \cap U' \ne \emptyset$. It follows that there exists an $\bb{A}^1$-fiber $A$ of $U$ such that $A \subseteq S$. Consider two distinct points $p, x \in A$ and let $y \in X^\reg$ be an arbitrary point distinct from $p$. By Theorem~\ref{FIT.th} the group $\SAut(X)$ acts infinitely transitively on $X^\reg$, hence there exists $g \in \SAut(X)$ such that $g \cdot x = y$ and $g \cdot p = p$. Then $p \in S \cap (g U')$ and the points $p, y$ belong to the same $\bb{A}^1$-fiber of the cylinder $g U'$. It follows that $y \in S$. Therefore, $S = X^\reg$ and $\cal{U}$ is transversal.

Conversely, assume that there is a transversal covering of $X^\reg$ by cylinders, which are principal in $X$. By utilizing the same argument as in the proof of Lemma~\ref{TCIF.le}, we see that $\SAut(X)$ acts on $X^\reg$ transitively, hence $X$ is flexible by Theorem~\ref{FIT.th}.
\end{proof}

Now we move to flexibility of affine cones over projective varieties. Given a projective variety $X$ and a very ample divisor $H$ on $X$, denote by $\AffCone_H(X)$ the affine cone over the embedding $X \into \bb{P}^m$ defined by $H$.

\begin{defi}
Let $X$ be a projective variety, $H$ be a very ample divisor on $X$, and $U \subseteq X$ be an open subset. The subset $U$ is called \emph{$H$-polar} if
$$
X \setminus U = \Supp D
$$
for some effective divisor $D$ on $X$ such that $D \sim d H$ for some positive integer $d$.
\end{defi}

Let $X$ be a normal projective variety with $\Pic(X) \cong \bb{Z}$ and $H$ be a very ample divisor on $X$. Then every open subset $U \subseteq X$ such that $X \setminus U$ is of pure codimension one in $X$ is $H$-polar. In particular, if $X$ is smooth and $U$ is affine, then $U$ is $H$-polar.

The following generalization of \cite[Theorem~5]{Pe13} about $1$-cylinders is obtained by analyzing the proof given in~\cite{Pe13}, utilizing Lemma~\ref{LTC.le}.(2), and using the results of \cite[Theorem~3.1.9.(b)]{KPZ11} and \cite[Lemma~2.8]{KPZ13}. Recall that a projective variety $X$ is called projectively normal with respect to a very ample divisor $H$ on $X$ if $\AffCone_H(X)$ is normal.

\begin{theo} \label{FAC.th}
Let $X$ be a normal projective variety and $H$ be a very ample divisor on~$X$. Let $\{U_i\}_{i \in I}$ be a transversal covering of $X^\reg$ by $H$-polar principal cylinders. Assume that either $X$ is smooth and $\Pic(U_i) = 0$ for all $i \in I$, or $X$ is projectively normal with respect to~$H$. Then $\AffCone_H(X)$ is flexible.
\end{theo}


\section{Affine cones over quadrics}
\label{ACOQ.se}


The fact that affine cones over quadrics are flexible seems to be known. For the lack of a reference, we provide a proof for this fact in this section.

The following simple observation is a technical tool that will be used in the current and the next sections.

\begin{lemm} \label{CTSQ.le}
Let $X \subseteq \bb{P}^n$ be a quadric of rank at least three and let $x \in X^\reg$. Then there exist coordinates $x_0, \dots, x_n$ in $\bb{P}^n$ such that
$$
x = (0 : 1 : 0 : \ldots : 0), \quad
\bb{T}_x X = \bb{V}_{\bb{P}^n}(x_0),
$$
and $X$ is given by a quadratic form
$$
f = x_0 x_1 + g(x_2, \dots, x_n).
$$
\end{lemm}

\begin{proof}
Choose coordinates $y_0, \dots, y_n$ in $\bb{P}^n$ such that
$$
x = (0 : 1 : 0 : \ldots : 0)
\insertText{and}
\bb{T}_x X = \bb{V}_{\bb{P}^n}(y_0).
$$
Then a direct computation shows that $X$ is given by a quadratic form
$$
f = y_0 h(y_0, \dots, y_n) + g(y_2, \dots, y_n)
$$
such that $h(0, 1, 0, \ldots, 0) = 1$. By setting
$$
x_0 = y_0, \quad x_1 = h(y_0, \dots, y_n), \insertText{and}
x_i = y_i \insertText{for} i \ge 2
$$
we obtain the required coordinates.
\end{proof}

The choice of coordinates as in Lemma~\ref{CTSQ.le} allows to find cylinders in a quadric.

\begin{lemm} \label{CQ.le}
Let $X \subseteq \bb{P}^n$ be a quadric of rank at least three and let $x \in X^\reg$. Then the subset $X \setminus \bb{T}_x X$ is an $(n - 1)$-cylinder in $X$, that is, $X \setminus \bb{T}_x X \cong \bb{A}^{n - 1}$.
\end{lemm}

\begin{proof}
Let $x_0, \dots, x_n$ be coordinates in $\bb{P}^n$ as in Lemma~\ref{CTSQ.le}. Then
$$
X \setminus \bb{T}_x X =
X \setminus \bb{V}_{\bb{P}^n}(x_0) \cong
\bb{V}_{\bb{A}^n} \big( \restrict{f}{x_0 = 1} \big) \cong
\bb{A}^{n - 1},
$$
since $\restrict{f}{x_0 = 1} = x_1 + g(x_2, \dots, x_n)$.
\end{proof}

\begin{lemm} \label{TCQ.le}
Let $X \subseteq \bb{P}^n$ be a quadric of rank at least three. Then the collection
$$
\cal{U} = \{ X \setminus \bb{T}_x X \}_{x \in X^\reg}
$$
is a transversal covering of $X^\reg$ by cylinders.
\end{lemm}

\begin{proof}
By Lemma~\ref{CTSQ.le} each cylinder in $\cal{U}$ is isomorphic to $\bb{A}^{n - 1}$, so $\cal{U}$ is clearly transversal. It remains to show that $\cal{U}$ covers $X^\reg$, that is, for each $x \in X^\reg$ there exists $y \in X^\reg$ such that $x \not\in \bb{T}_y X$. Let $x_0, \dots, x_n$ be coordinates in $\bb{P}^n$ as in Lemma~\ref{CTSQ.le}. Then for $y = (1 : 0 : \ldots : 0)$ we have $\bb{T}_y X = \bb{V}_{\bb{P}^n}(x_1)$ and $$
x = (0 : 1 : 0 : \ldots : 0) \not\in \bb{V}_{\bb{P}^n}(x_1)
$$
as required.
\end{proof}

\begin{prop} \label{ACOQF.pr}
Let $X \subseteq \bb{P}^n$ be a quadric of rank at least three. Then every affine cone over $X$ is flexible.
\end{prop}

\begin{proof}
Denote by $r$ the rank of the quadric $X$. If $(n, r) \ne (3, 4)$, then $\Pic(X) \cong \bb{Z}$ so the assertion immediately follows from Lemma~\ref{TCQ.le} and Theorem~\ref{FAC.th}. Assume that $(n, r) = (3, 4)$, that is, $X$ is a smooth quadric in $\bb{P}^3$. Then $X \cong \bb{P}^1 \times \bb{P}^1$ and there is a bidegree isomorphism
$$
\Pic(\bb{P}^1 \times \bb{P}^1) \cong \bb{Z}^2.
$$
The complement in $X$ of cylinders constructed in Lemma~\ref{CQ.le} are supported at a divisor in the class $(1, 1) \in \bb{Z}^2$. Therefore, since the classes of ample divisors on $X$ are of the form $(a, b) \in \bb{Z}^2$ with $a, b > 0$, the assertion again follows from Lemma~\ref{TCQ.le} and Theorem~\ref{FAC.th}.
\end{proof}

\begin{rema}
Let $X \subseteq \bb{P}^n$ be a quadric of rank at least three. The fact that the affine cone over this standard embedding is flexible also follows from flexibility of suspensions, see \cite[Section~5]{KZ99} and \cite[Theorem~3.2]{AZK12}.
\end{rema}


\section{Complements of quadrics in projective spaces}
\label{CQPS.se}


Recall that if $X \subseteq \bb{P}^n$ is a hypersurface, then the complement $\bb{P}^n \setminus X$ is an affine variety.

Let us present a construction of cylinders in the complement of a smooth quadric in a projective space. We utilize an argument, similar to the one from \cite[Lemma~5.21]{BPS24}.

\begin{lemm} \label{CCQ.le}
Let $X \subseteq \bb{P}^n$, $n \ge 2$ be a smooth quadric. Consider a point $x \in X$ and denote $Y = \bb{P}^n \setminus X$. Then the subset $Y \setminus \bb{T}_x X$ is a principal $(n - 1)$-cylinder in $Y$.
\end{lemm}

\begin{proof}
Let $x_0, \dots, x_n$ be coordinates in $\bb{P}^n$ as in Lemma~\ref{CTSQ.le}. Then
$$
Y \setminus \bb{T}_x X =
\bb{D}_Y \left( \frac{x_0^2}{f} \right) =
\bb{D}_{\bb{P}^n}(x_0 f) \cong
\bb{D}_{\bb{A}^n} \big( \restrict{f}{x_0 = 1} \big) \cong
\bb{A}^1_* \times \bb{A}^{n - 1},
$$
where the last isomorphism is given by the formulas
$$
\begin{array}{rcl}
(x_1, \dots, x_n) & \mapsto &
\big( f(1, x_1, \dots, x_n), \ x_2, \ \dots, \ x_n \big)
\\
\big( t - g(x_2, \dots, x_n), \ x_2, \ \dots, \ x_n \big) &
\mapsfrom & (t, x_2, \dots, x_n)
\end{array}
.
$$
\end{proof}

Let $X \subseteq \bb{P}^n$, $n \ge 2$ be a smooth quadric and denote $Y = \bb{P}^n \setminus X$. We want to show that cylinders constructed in Lemma~\ref{CCQ.le} define a transversal covering of $Y$. We treat separately the two cases according to the parity of $n$.

First consider the case of odd $n = 2 m - 1$ and choose coordinates $x_1, \dots, x_m, y_1, \dots, y_m$ in $\bb{P}^n$ such that $X$ is given by a quadratic form
$$
f = x_1 y_1 + \ldots + x_m y_m.
$$
We will write coordinates of a point in $\bb{P}^n$ in the following order:
$$
(x_1 : y_1 : x_2 : y_2 : \ldots : x_n : y_n).
$$
Fix an index $i \in \{1, \dots, m\}$. Consider a point $p_i \in X$ given by the equations
$$
y_i = 1, \quad x_i = 0, \insertText{and}
y_j = x_j = 0 \insertText{for} j \ne i.
$$
Then $\bb{T}_{p_i} X = \bb{V}_{\bb{P}^n}(x_i)$. Denote $U_i = Y \setminus \bb{T}_{p_i} X$ and
$$
\sigma_i =
x_1 y_1 + \ldots + \widehat{x_i y_i} + \ldots + x_m y_m =
f - x_i y_i.
$$
Then
$$
U_i = \bb{D}_{\bb{P}^n}(x_i f) \cong
\bb{D}_{\bb{A}^n}(\restrict{f}{x_i = 1}) \cong
\bb{A}^1_* \times \bb{A}^{n - 1},
$$
where the last isomorphism is given by the formulas
$$
\begin{array}{rcl}
(x_1, y_1, \dots, \hat{x}_i, y_i, \dots, x_m, y_m) &
\mapsto &
(\restrict{f}{x_i = 1}, x_1, y_1, \dots,
\hat{x}_i, \hat{y}_i, \dots, x_m, y_m)
\\
(x_1, y_1, \dots, \hat{x}_i, t - \sigma_i, \dots, x_m, y_m) &
\mapsfrom &
(t, x_1, y_1, \dots, \hat{x}_i, \hat{y}_i, \dots, x_m, y_m)
\end{array}
.
$$
Therefore, for each $t \in \bb{A}^1_*$ there is a fiber $A_i(t)$ of the cylinder $U_i$ consisting of points of the form
$$
(a_1 : b_1 : \ldots : a_{i - 1} : b_{i - 1} :
1 : t - \sigma_i : a_{i + 1} : b_{i + 1} : \ldots : a_m : b_m)
$$
for $a_j, b_j \in \bb{K}$.

\begin{lemm} \label{ODTC.le}
The collection $\cal{U} = \{U_1, \dots, U_m\}$ of cylinders in $Y$ is a transversal covering.
\end{lemm}

\begin{proof}
Clearly, for every point $p \in Y$ at least one of the coordinates $x_i$ of $p$ is nonzero. Therefore, $\cal{U}$ covers $Y$, so it remains to show that $\cal{U}$ is transversal. Consider a nonempty subset $S \subseteq Y$ which is invariant with respect to $\cal{U}$. Our goal is to show that $S = Y$. Take a point $p \in S$.

It is easy to see that there exists $i \in \{1, \dots, m\}$ such that both coordinates $x_i$ and $y_i$ of $p$ are nonzero. By symmetry, it suffices to consider the case of $i = 1$. We may assume that $x_1 = 1$ and $p \in A_1(t)$ for some $t \in \bb{A}^1_*$. By choosing another point in the fiber $A_1(t)$, we may assume that
$$
p = (1 : t : 0 : \ldots : 0).
$$
Consider a scalar $\lambda \in \bb{K}^\times$. By moving $p$ along $A_1(t)$, we conclude that $S$ also contains a point
$$
(1 : t : \lambda : 0 : \ldots : 0) =
\left( \frac{1}{\lambda} : \frac{t}{\lambda} : 1 : 0 : \ldots :
0 \right) \in A_2 \left( \frac{t}{\lambda^2} \right).
$$
Moving along $A_2(t \lambda^{-2})$, we obtain a point
$$
\left( 1 : \frac{t}{\lambda} : 1 :
\frac{t}{\lambda^2} - \frac{t}{\lambda} : 0 : \ldots :
0 \right) \in A_1 \left( \frac{t}{\lambda^2} \right),
$$
and moving along $A_1(t \lambda^{-2})$, we get a point
$$
\left( 1 : \frac{t}{\lambda^2} : 0 : \ldots : 0 \right).
$$
Therefore, $S$ contains all points of the form $(1 : t : 0 : \ldots : 0)$, where $t \in \bb{K}^\times$. So $S$ contains $U_1$. For each $i \in \{1, \dots, m\}$ we have $U_1 \cap U_i \ne \emptyset$, hence by symmetry $S$ contains $U_i$ for all $i \in \{1, \dots, m\}$. We conclude that $S = Y$.
\end{proof}

Now consider the case of even $n = 2 m$ and choose coordinates $x_1, \dots, x_m, y_1, \dots, y_m, z$ in $\bb{P}^n$ such that $X$ is given by a quadratic form
$$
f = x_1 y_1 + \ldots + x_m y_m + z^2.
$$
We will write coordinates of a point in $\bb{P}^n$ in the following order:
$$
(x_1 : y_1 : x_2 : y_2 : \ldots : x_n : y_n : z).
$$
Fix an index $i \in \{1, \dots, m\}$. Similarly to the case of odd $n$, consider a point $p_i \in X$ given by the equations
$$
y_i = 1, \quad x_i = 0, \quad
y_j = x_j = 0 \insertText{for} j \ne i, \insertText{and}
z = 0.
$$
Then $\bb{T}_{p_i} X = \bb{V}_{\bb{P}^n}(x_i)$. Denote $U_i = Y \setminus \bb{T}_{p_i} X$ and
$$
\sigma_i =
x_1 y_1 + \ldots + \widehat{x_i y_i} + \ldots + x_m y_m + z^2 =
f - x_i y_i.
$$
Then for each $t \in \bb{A}^1_*$ there is a fiber $A_i(t)$ of the cylinder $U_i$ consisting of points of the form
$$
(a_1 : b_1 : \ldots : a_{i - 1} : b_{i - 1} :
1 : t - \sigma_i :
a_{i + 1} : b_{i + 1} : \ldots : a_m : b_m : c)
$$
for $a_j, b_j, c \in \bb{K}$.

We will also require another similar collection of cylinders. Namely, define points $q_i \in X$ and cylinders $V_i$ in $Y$ by interchanging $x_i$ with $y_i$ in the definitions above. Then for each $t \in \bb{A}^1_*$ there is a fiber $B_i(t)$ of the cylinder $V_i$ consisting of points of the form
$$
(a_1 : b_1 : \ldots : a_{i - 1} : b_{i - 1} :
t - \sigma_i : 1 :
a_{i + 1} : b_{i + 1} : \ldots : a_m : b_m : c)
$$
for $a_j, b_j, c \in \bb{K}$.

Consider a point $r = (0 : \ldots : 0 : 1) \in Y$. Note that $r \not\in U_i \cup V_i$ for all $1 \le i \le m$, hence we will need at least one more cylinder. Let $q \in X$ be a point such that $r \not \in \bb{T}_q X$: for example, one can take $q = (0 : \ldots 0 : -1 : 1 : 1)$. Denote $W = Y \setminus \bb{T}_q X$.

\begin{lemm} \label{EDTC.le}
The collection $\cal{U} = \{U_1, V_1, \dots, U_m, V_m, W\}$ of cylinders in $Y$ is a transversal covering.
\end{lemm}

\begin{proof}
Denote $Y' = Y \setminus \{r\}$ and $\cal{U}' = \cal{U} \setminus \{W\}$. Clearly, for every point $p \in Y'$ at least one of the coordinates $x_i$ or $y_i$ of $p$ is nonzero. Therefore, $\cal{U}'$ covers $Y'$. It follows that $\cal{U}$ covers $Y$. We conclude that it remains to show that $\cal{U}'$ is a transversal covering of $Y'$. Consider a nonempty subset $S \subseteq Y'$ which is invariant with respect to $\cal{U}'$. Our goal is to show that $S = Y'$. Take a point $p \in S$.

Consider $i \in \{1, \dots, m\}$ such that either coordinate $x_i$ or $y_i$ of $p$ is nonzero. By symmetry, we can assume that $i = 1$ and $x_1 = 1$. Then $p \in A_1(t)$ for some $t \in \bb{A}^1_*$, so after moving along the fiber $A_1(t)$ we may assume that
$$
p = (1 : 0 : \ldots : 0 : \lambda),
$$
where $\lambda^2 = t$.

First, let us show that $S$ contains a point $s_\mu = (1 : 0 : \ldots : 0 : \mu)$ for every $\mu \in \bb{K}^\times$. Move $p$ along $A_1(\lambda^2)$ to get
$$
(1 : -3 \lambda^2 : 0 : \ldots : 0 : 2 \lambda) =
\left(
    -\frac{1}{3 \lambda^2} : 1 : 0 : \ldots : 0 :
    -\frac{2}{3 \lambda}
\right) \in
B_1 \left( \frac{1}{9 \lambda^2} \right).
$$
Consider a scalar $c \in \bb{K}$ such that $c^2 \ne \frac{1}{9 \lambda^2}$. By moving $p$ along $B_1 \big( (9 \lambda^2)^{-1} \big)$, we obtain a point
$$
\left(
    \frac{1 - 9 c^2 \lambda^2}{9 \lambda^2} : 1 : 0 : \ldots :
    0 : c
\right) =
\left(
    1 : \frac{9 \lambda^2}{1 - 9 c^2 \lambda^2} : 0 : \ldots :
    0 : \frac{9 c \lambda^2}{1 - 9 c^2 \lambda^2}
\right),
$$
which can be moved by a fiber of $U_1$ to a point $(1 : 0 : \ldots : 0 : b)$, where
$$
b^2 =
\frac{9 \lambda^2}{1 - 9 c^2 \lambda^2} +
\left( \frac{9 c \lambda^2}{1 - 9 c^2 \lambda^2} \right)^2 =
\frac{
    9 \lambda^2
}{
    (1 - 9 c^2 \lambda^2)^2
}
.
$$
The right-hand side is a function of $c$ with image equal to $\bb{K}^\times$, hence $b$ can be an arbitrary nonzero scalar.

Since $s_\mu \in S$ for all $\mu \in \bb{K}^\times$, we conclude that $U_1 \subseteq S$. By symmetry we have $U_i \cup V_i \subseteq S$ for all $1 \le i \le m$, hence $S = Y'$ as required.
\end{proof}

Now we construct a transversal covering by cylinders of the complement of an arbitrary quadric of rank at least three. For this we present a bit more general argument.

We call two linear subspaces $K, L \subseteq \bb{P}^n$ complementary if $K \cap L = \emptyset$ and
$$
n = \dim K + \dim L + 1.
$$
We say that coordinates $x_0, \dots, x_n$ in $\bb{P}^n$ are compatible with $K$ and $L$ if $x_0, \dots, x_{\dim K}$ are coordinates in $K$ and $x_{\dim K + 1}, \dots, x_n$ are coordinates in $L$. In other words,
$$
K = \bb{V}_{\bb{P}^n}(x_{\dim K + 1}, \dots, x_n)
\insertText{and}
L = \bb{V}_{\bb{P}^n}(x_0, \dots, x_{\dim K}).
$$

\begin{defi}
Consider a projective space $\bb{P}^n$, complementary subspaces $K, L \subseteq \bb{P}^n$ and a projective variety $X_0 \subseteq K$. For each $x \in X_0$ denote $L_x = \gen{x, L}_{\bb{P}^n}$. The projective variety $X = \bigcup_{x \in X_0} L_x$ is called a \emph{projective cone} with a \emph{base} $X_0$ and a \emph{vertex} $L$.
\end{defi}

Note that if $x_0, \dots, x_n$ are coordinates in $\bb{P}^n$ compatible with $K$ and $L$ and $X_0 \subseteq K$ is given by equations $h_1 = \ldots = h_m = 0$ for some homogeneous polynomials $h_i \in \bb{K}[x_0, \dots, x_{\dim K}]$, then the projective cone $X$ is given in $\bb{P}^n$ by the same equations $h_1 = \ldots = h_m = 0$.

Let $K$ and $L$ be two complementary linear subspaces in $\bb{P}^n$. For a closed subset $W \subseteq K$ denote by $\wh{W}$ the projective cone with a base $W$ and a vertex $L$. For an open subset $U \subseteq K$ denote $\hat{U} = \bb{P}^n \setminus \wh{W}$, where $W = K \setminus U$.

Denote $\dim K = k$, $\dim L = l$ and choose coordinates $x_0, \dots, x_k, y_0, \dots, y_l$ in $\bb{P}^n$ compatible with $K$ and $L$. Let $H \subseteq K$ be a hyperplane given by a linear form $h \in \bb{K}[x_0, \dots, x_k]$. Assume that $U \subseteq K$ is an open subset such that $H \subseteq K \setminus U$. Then $\hat{U}$ is an $(l + 1)$-cylinder, namely
$$
\begin{array}{rcl}
\hat{U} & \cong & U \times \bb{A}^{l + 1}
\\
(x_0 : \ldots : x_k : y_0 : \ldots : y_l) & \mapsto &
\left(
    (x_0 : \ldots : x_k),
    \left( \frac{y_0}{h}, \dots, \frac{y_l}{h} \right)
\right)
\\
(x_0 : \ldots : x_k : h t_0 : \ldots : h t_l) & \mapsfrom &
\big( (x_0 : \ldots : x_k), (t_0 : \ldots : t_l) \big)
\end{array}
.
$$

\begin{lemm} \label{TCCC.le}
Let $X_0 \subseteq K$ be a projective variety. Denote $Y_0 = K \setminus X_0$ and $X = \hat{X}_0$. Assume that there is a transversal covering $\cal{U} = \{U_i\}_{i \in I}$ of $Y_0$ by cylinders such that the complement $K \setminus U_i$ contains a hyperplane in $K$ for all $i \in I$. Then $\hat{\cal{U}} = \{\hat{U}_i\}_{i \in I}$ is a transversal covering of $Y = \bb{P}^n \setminus X$.
\end{lemm}

\begin{proof}
For $x \in K$ denote $L'_x = L_x \setminus L$, where $L_x = \gen{x, L}_{\bb{P}^n}$. For every point $p \in \bb{P}^n \setminus L$ there exists a unique point $x \in K$ such that $p \in L'_x$, and, moreover, for distinct $x, y \in K$ the subsets $L'_x$ and $L'_y$ of $\bb{P}^n$ are disjoint. We can write
$$
Y = \bigsqcup_{x \in Y_0} L'_x \insertText{and}
\hat{U}_i = \bigsqcup_{x \in U_i} L'_x
\insertText{for all} i \in I.
$$
In particular, it follows that $\hat{\cal{U}}$ is a covering of $Y$.

Consider a nonempty subset $S \subseteq Y$, which is invariant with respect to $\hat{\cal{U}}$. Take a point $p \in S$ and let $x \in Y_0$ be such that $p \in L'_x$. By moving $p$ along an $\bb{A}^{l + 1}$-fiber of a cylinder in $\hat{\cal{U}}$, we see that $x \in L'_x \subseteq S$. By moving $x$ along a fiber of a cylinder in $\cal{U}$, we conclude that $Y_0 \subseteq S$. It follows that $S = Y$.
\end{proof}

Now we are ready to state the main results of this section.

\begin{prop} \label{TCCQ.pr}
Let $X \subseteq \bb{P}^n$ be a quadric of rank at least three. Then the complement $Y = \bb{P}^n \setminus X$ admits a transversal covering by principal cylinders.
\end{prop}

\begin{proof}
If $X$ is smooth, then the assertion was proved in Lemmas~\ref{ODTC.le} and~\ref{EDTC.le}. Assume that $X$ is singular and denote by $r$ the rank of $X$. Then $X$ is a projective cone with a base $X_0$, where $X_0$ is a smooth quadric in $\bb{P}^{r - 1}$. One can easily see that the cylinders $\hat{U}_i$, constructed in Lemma~\ref{TCCC.le}, are principal provided that the cylinders $U_i$ are principal. Therefore, the assertion follows from Lemma~\ref{TCCC.le}.
\end{proof}

\begin{theo} \label{CQF.th}
Let $X \subseteq \bb{P}^n$ be a quadric of rank at least three. Then the complement $Y = \bb{P}^n \setminus X$ is flexible.
\end{theo}

\begin{proof}
Combine Propositions~\ref{CF.pr} and~\ref{TCCQ.pr}.
\end{proof}

\begin{rema}
One can also deduce flexibility of the complement $Y = \bb{P}^n \setminus X$ of a smooth quadric $X \subseteq \bb{P}^n$, $n \ge 2$ from the following observation. A semisimple algebraic group $G = \SO{n + 1}$ acts on $Y$ transitively. By~\cite[Lemma~1.1]{Po11} the group $G$ is generated by its $\Ga$-subgroups. Therefore, $Y$ is flexible by Theorem~\ref{FIT.th}.
\end{rema}


\section{Smooth complete intersections of two quadrics}
\label{SCITQ.se}


In this section we prove flexibility of affine cones over a smooth complete intersection of two quadrics in $\bb{P}^{n + 2}$, $n \ge 3$.

Let $X$ be a smooth complete intersection of two quadrics $Q_1$ and $Q_2$ in $\bb{P}^{n + 2}$, $n \ge 3$. Recall that $X$ is a Fano variety with Picard number $1$ and index $n - 1$.

First, let us recall a construction from \cite[Lemma~3.2]{CPPZ21} that provides open subsets in $X$ isomorphic to the complement of a quadric in $\bb{P}^n$. For a line $l$ in $X$ we will denote by $D_l(X)$ a prime divisor in $X$ swept out by lines in $X$ meeting $l$. Fix a line $l \subseteq X$ and denote $D = D_l(X)$. Let $\sigma \colon \tilde{X} \to X$ be the blowup of $l$ and $\psi \colon X \dashrightarrow \bb{P}^n$ be the projection from $l$. Then there exists a commutative diagram
$$
\begin{tikzcd}[column sep=small]
 & & & \tilde{D} \ar[phantom, "\subset"]{r} &
\tilde{X} \ar{lld}{\sigma} \ar{rrd}{\phi} &
E \ar[phantom, "\supset"]{l}
\\
l \ar[phantom, "\subset"]{r} & D \ar[phantom, "\subset"]{r} &
X \ar[dashed]{rrrr}{\psi} & & & &
\bb{P}^n & \phi(E) \ar[phantom, "\supset"]{l}
\end{tikzcd}
,
$$
where $\tilde{D}$ is the proper transform of $D$ and $E$ is the exceptional divisor of $\sigma$. In other words, $\phi$ is a resolution of indeterminacy of $\psi$. The map $\psi$ is birational and it induces an isomorphism $X \setminus D \cong \bb{P}^n \setminus \phi(E)$. Moreover, $\phi(E) \subseteq \bb{P}^n$ is a quadric of rank $r \ge 3$. In fact, $\phi(E)$ contains a one-parameter family of linear subspaces of dimension $n - 2$. By \cite[Section~1.1]{Re72} the maximal dimension of a linear subspace contained in $\phi(E)$ is equal to
$$
n + 1 - r + \left[ \frac{r - 2}{2} \right] =
n - r + \left[ \frac{r}{2} \right] =
n - \left\lceil \frac{r}{2} \right\rceil.
$$
Therefore,
$$
n - 2 \le n - \left\lceil \frac{r}{2} \right\rceil,
\insertText{so} r \le 4.
$$

Consider a vector space $V$ over $\bb{K}$ such that $\bb{P}^{n + 2} = \bb{P}(V)$. Choose two quadratic forms $\beta$ and $\gamma$ on $V$ such that $Q_1$ and $Q_2$ are given by equations $\beta = 0$ and $\gamma = 0$ respectively. By abuse of notation, we denote by the same letters $\beta$ and $\gamma$ the corresponding symmetric bilinear forms on $V$. Note that for a point $p \in X$ the projective tangent space $\bb{T}_p X \subseteq \bb{P}^{n + 2}$ is given by equations
$$
\beta(p, x) = \gamma(p, x) = 0.
$$

\begin{lemm} \label{LSX.le}
Consider a subset $W \subseteq \bb{P}^{n + 2}$. Then $\gen{W}_{\bb{P}^{n + 2}} \subseteq X$ if and only if
$$
\beta(w_1, w_2) = \gamma(w_1, w_2) = 0
$$
for all $w_1, w_2 \in W$.
\end{lemm}

\begin{proof}
For each $w \in W$ choose a vector $\tilde{w} \in V$ that spans the line $w$.

Assume that $\gen{W}_{\bb{P}^{n + 2}} \subseteq X$. Then clearly $\beta(w, w) = \gamma(w, w) = 0$ for all $w \in W$. Consider two elements $w_1, w_2 \in W$ and denote by $u$ the element of $\gen{W}_{\bb{P}^{n + 2}}$ spanned by a vector $\tilde{u} = \tilde{w}_1 + \tilde{w}_2$. Then
$$
0 = \beta(u, u) =
\beta(\tilde{w}_1 + \tilde{w}_2, \tilde{w}_1 + \tilde{w}_2) =
2 \beta(\tilde{w}_1, \tilde{w}_2),
$$
hence $\beta(w_1, w_2) = 0$. The argument for $\gamma$ is identical.

Assume the converse. Consider $u \in \gen{W}_{\bb{P}^{n + 2}}$ spanned by a vector
$$
\tilde{u} = \lambda_1 \tilde{w}_1 + \ldots +
\lambda_k \tilde{w}_k
$$
for some $w_1, \dots, w_k \in W$ and $\lambda_1, \dots, \lambda_k \in \bb{K}$. Then
$$
\beta(\tilde{u}, \tilde{u}) = \beta \left(
\sum_{i = 1}^k \lambda_i \tilde{w}_i, \
\sum_{i = 1}^k \lambda_i \tilde{w}_i
\right) =
\sum_{i, j = 1}^k \lambda_i \lambda_j
\beta(\tilde{w}_i, \tilde{w}_j) = 0,
$$
so $u \in Q_1$. The argument for $Q_2$ is identical.
\end{proof}

Consider a point $p \in X$. Denote by $S_p(X)$ the union of lines in $X$ containing $p$. We claim that
$$
S_p(X) = \bb{T}_p X \cap X.
$$
The inclusion $S_p(X) \subseteq \bb{T}_p X \cap X$ is clear. Conversely, let $x \in \bb{T}_p X \cap X$. Then $x \in S_p(X)$ by Lemma~\ref{LSX.le}. In particular, $\dim S_p(X) \ge 1$, hence for every point $p \in X$ there exists a line in $X$ containing $p$.

Denote by $L(X)$ the set of lines in $X$, that is, a closed subset in $\Gr{2}{n + 3}$. For a point $p \in X$ denote by $L_p(X)$ the set of lines in $X$ containing $p$, that is, a closed subset in $L(X)$. Clearly, $\dim L_p(X) = \dim S_p(X) - 1$.

For a line $l$ in $X$ consider an open subset $U_l(X) = X \setminus D_l(X)$ in $X$. Denote
$$
\cal{U} = \{U_l(X)\}_{l \in L(X)}.
$$

\begin{lemm} \label{CX.le}
The collection $\cal{U}$ is a covering of $X$.
\end{lemm}

\begin{proof}
Fix a point $x \in X$. Our goal is to show that there exists a line $l$ in $X$ such that $x \not\in D_l(X)$.

First assume that $n = 3$. In this case $\dim S_x(X) = 1$, so there are finitely many points in $L_x(X)$, name them $h_1, \dots, h_m$. Take a point $y \in X$ which does not belong to the divisor $D_{h_1}(X) \cup \ldots \cup D_{h_m}(X)$. Then for any line $l$ in $X$ which contains $y$ we have $x \not\in D_l(X)$. Indeed, if $x \in D_l(X)$, then there is a line $k$ in $X$ such that $x \in k$ and $k$ meets $l$. It follows that $k = h_i$ for some $i \in \{1, \dots, m\}$ and $y \in D_{h_i}(X)$, a contradiction.

Now assume that $n \ge 4$. Take a general $5$-dimensional subspace $W \subseteq \bb{P}^{n + 2}$ containing~$x$. Then $Y = W \cap X$ is a smooth complete intersection of two quadrics in $W \cong \bb{P}^5$, hence there exists a line $l$ in $Y$ such that $x \not\in D_l(Y)$. Since $D_l(Y) = D_l(X) \cap W$, we conclude that $x \not\in D_l(X)$.
\end{proof}

\begin{theo} \label{FACX.th}
Let $X$ be a smooth complete intersection of two quadrics in $\bb{P}^{n + 2}$, $n \ge 3$. Then every affine cone over $X$ is flexible.
\end{theo}

\begin{proof}
Since $\Pic(X) \cong \bb{Z}$, by~Theorem~\ref{FAC.th} it suffices to show that there exists a transversal covering of $X$ by cylinders. By Lemma~\ref{CX.le}, there exists a covering $\cal{U} = \{U_i\}_{i \in I}$ of $X$ by open subsets, where $U_i$ is isomorphic to the complement of a quadric of rank at least three in $\bb{P}^n$ for all $i \in I$. By Proposition~\ref{TCCQ.pr} each $U_i$ admits a transversal covering by cylinders. The assertion now follows from Lemma~\ref{LTC.le}.(1).
\end{proof}

Let $U \subseteq X$ be a cylinder in the transversal covering of $X$, constructed in the proof of Theorem~\ref{FACX.th}. We see that the complement $X \setminus U$ is the union of two prime divisors: a hyperplane section and a divisor of the form $D_l(X)$. Therefore, since $D_l(X)$ is cut out by a quadric, we have $X \setminus U = \Supp D$, where $D \sim 3 H$ and $H$ is a hyperplane section of $X$.

\begin{coro} \label{EACX.co}
Consider an affine space $\bb{A}^n$, $n \ge 6$ with coordinates $x_1, \dots, x_n$. Let $X \subseteq \bb{A}^n$ be an affine variety given by equations
$$
x_1^2 + \ldots + x_n^2 = 0 \insertText{and}
\lambda_1 x_1^2 + \ldots + \lambda_n x_n^2 = 0,
$$
where $\lambda_1, \dots, \lambda_n \in \bb{K}$ are pairwise distinct. Then $X$ is flexible.
\end{coro}

\begin{proof}
Use \cite[Proposition~2.1]{Re72} and Theorem~\ref{FACX.th}.
\end{proof}

Note that the variety $X$ in Corollary~\ref{EACX.co} is normal, since a smooth complete intersection of two quadrics is projectively normal.

It remains an interesting question whether flexibility holds for affine cones over complete intersections of three quadrics, or for mildly singular degenerations of the varieties considered above.


{}


\end{document}